\documentclass[12pt]{amsart}
\usepackage[left=1.5in, right=1.5in, top=1.5in, bottom=1.5in]
{geometry}
\newtheorem{thm}{Theorem}[section]
\newtheorem{cor}[thm]{Corollary}

\newtheorem{prop}[thm]{Proposition}
\theoremstyle{definition}

\newtheorem{rem}[thm]{Remark}
\usepackage{graphicx}
\usepackage{amssymb}
\usepackage{amsmath}
\usepackage{epstopdf}
\usepackage{hyperref}
\usepackage[colorinlistoftodos,prependcaption,textsize=tiny,
textwidth=0.75in, color=cyan]{todonotes}
\title[]{The nonlocal Almgren problem}
\author{Emanuel Indrei}
\address{Department of Mathematics\\
Kennesaw State University\\
Marietta, GA 30060\\
USA.}
\begin{document}
\setcounter{page}{1}
\pagenumbering{arabic}
\maketitle
\begin{abstract}
In the nonlocal Almgren problem, the goal is to investigate the convexity of a minimizer under a mass constraint via a nonlocal free energy generated with  some nonlocal perimeter and convex potential. In the paper, the main result is a quantitative stability theorem for the nonlocal free energy assuming symmetry on the potential. In addition, several results that involve uniqueness, non-existence, and moduli estimates
from the theory for crystals are proven also  in the nonlocal context. 
\end{abstract}
\section{Introduction}
A fundamental theorem in real analysis is Taylor's theorem: assume \\
\noindent $f:\mathbb{R} \rightarrow \mathbb{R}$ is twice differentiable, 

$$
f(e)=f(e_m)+f'(e_m)(e-e_m)+\frac{f''(\alpha_{e,e_m})}{2}|e-e_m|^2,
$$
\\
$\alpha_{e,e_m} \in (e_m,e)$. 
If $f'(e_m)=0$, $f''(\alpha_{e,e_m})\ge a_*>0$, then

\begin{equation} \label{sqt}
f(e)-f(e_m) \ge\frac{a_*}{2}|e-e_m|^2.
\end{equation}

In particular, one may obtain sharp information  from this: assume $f$ is a quantity which classifies an optimizer $e_m$ in the sense that $f(e)=f(e_m)$ implies $e=e_m$. Now, let $e$ satisfy
$f(e)=f(e_m)+ \epsilon$, with $\epsilon>0$ a small number; thus $f(e) \approx f(e_m)$ and the expectation is that $e \approx e_m$. The utility of \eqref{sqt} is to make this clear
in the context that $e$ is at most, up to a constant, $ \sqrt{\epsilon}$ from $e_m$.

When instead of a real-valued function $f$ the object of investigation is an energy $\mathcal{E}$ which is defined on measurable sets $E$, recent work has investigated analogous estimates. An application is to understand the perturbations of minimizers through the energy.
The main problem is to minimize $\mathcal{E}$ subject to a mass constraint $|E|=m$. Thus at the minimum $E_m$, the first variation is zero $\mathcal{E}'(E_m)=0$, hence if some lower bound exists on the second variation, it is natural to anticipate that mod an invariance class

\begin{equation} \label{es}
\mathcal{E}(E) - \mathcal{E}(E_m)\ge a_*||\chi_E-\chi_{E_m}||^2,
\end{equation}
\\
where $\chi_E$ is the characteristic function of $E$. A natural norm  is often  chosen to be the $L^1$ norm \cite{zbMATH06728315}.
In applications, the free energy is of the form
$$
\mathcal{E} (E)=\mathcal{E}_{s}(E)+\mathcal{E}_{a}(E),
$$
where $\mathcal{E}_{s}$ is a surface energy and $\mathcal{E}_{a}$ a potential/repulsion energy. 
In the next discussion, four physical and fundamental energies are underscored.

\subsection{The Free Energy}
The crystal theory starts with the anisotropic surface energy on sets of finite perimeter $E \subset \mathbb{R}^n$ with reduced boundary $\partial^* E$:
$$
\mathcal{E}_{s}(E)=\mathcal{F}(E)=\int_{\partial^* E} f(\nu_E) d\mathcal{H}^{n-1},
$$
where $f$ is a surface tension, i.e. a convex positively 1-homogeneous 

$$f:\mathbb{R}^n\rightarrow [0,\infty)$$
\\
with $f(x)>0$ if $|x|>0$.
The potential energy of a set $E$ is
$$
\mathcal{E}_{a}(E)=\mathcal{G}(E)=\int_E g(x)dx,
$$
where $g \ge 0$, $g(0)=0$, $g \in L_{loc}^\infty$  \cite{MR4730410, qk, qkv, pFZ, MR493671, MR872883,MR1116536, MR1130601, MR2672283, MR4380032, MR2807136}; see in addition many interesting references in \cite{MR4730410} that comprehensively discuss the history. In thermodynamics, to obtain a crystal, one minimizes the free energy 
$$
\mathcal{E}(E)=\mathcal{F}(E)+\mathcal{G}(E)   
$$
under a mass constraint. Gibbs and Curie independently discovered this physical principle \cite{G, Crist}.\\

\subsection{The Binding Energy}
\noindent The nonlocal Coulomb repulsion energy is given via
$$
\mathcal{E}_{a}(E)=\mathcal{D}(E)=\alpha_1 \int \int_{E \times E} \frac{1}{|z-y|^\lambda}dz dy,
$$
$\lambda \in (0,n)$, $\alpha_1>0$.
The binding energy of a set of finite perimeter $E \subset \mathbb{R}^n$ is the sum
$$
\mathcal{E}(E)=\mathcal{F}(E)+\mathcal{D}(E).     
$$
In the classical context, $\lambda=1$, $f(x)=|x|$, $n=3$, $\alpha_1=\frac{1}{2}$ \cite{MR4314139}. The theory is historically attributed to
Gamow via  his 1930 paper \cite{row52} and it successfully predicts the non-existence of nuclei with a large atomic number, cf. references in \cite{qkvp, qkvk, MR3322379}. \\

\subsection{The  Nonlocal  Free Energy}
\noindent The nonlocal perimeter encodes a parameter $\alpha \in (0,1)$

$$
\mathcal{E}_{s}(E)=P_\alpha(E)=\int_{E}\int_{E^c}  \frac{1}{|x-y|^{\alpha+n }}dxdy
$$
\cite{MR2425175, MR3322379, MR2675483, MR4674821}. Caffarelli, Roquejoffre, and Savin investigated the Plateau problem with respect to the nonlocal energy functionals \cite{MR2675483}. The nonlocal isoperimetric inequality has appeared in \cite{MR2425175}: assume $|E|=|B_a|$, then 
$$
P_\alpha(E)\ge P_\alpha(B_a) 
$$
with equality if and only if $E=B_a+x$. Hence the nonlocal free energy is

$$
\mathcal{E} (E)=P_\alpha(E)+\mathcal{G}(E),
$$
\cite{MR3640534, MR4674821}.

\subsection{The  Nonlocal Binding Energy}
The nonlocal Coulomb repulsion energy together with the nonlocal perimeter in the aforementioned give the nonlocal binding energy
$$
\mathcal{E} (E)=P_\alpha(E)+\mathcal{D}(E),
$$
see \cite{MR3322379} for a theorem on the minimizers when the mass is small.

\subsection{The Main Problem}

\noindent Observe via the above that four main choices of $\mathcal{E}$ are:

$$
\mathcal{E} (E)=\mathcal{F}(E)+\mathcal{G}(E)
$$

$$
\mathcal{E} (E)=\mathcal{F}(E)+\mathcal{D}(E)
$$

$$
\mathcal{E} (E)=P_\alpha(E)+\mathcal{G}(E)
$$

$$
\mathcal{E} (E)=P_\alpha(E)+\mathcal{D}(E).
$$
\\

Therefore the central problem is: assume $m>0$ and solve
$$
\inf\{\mathcal{E}(E): |E|=m\}.
$$ 
Naturally, the questions involve existence, uniqueness, convexity, and optimal stability.
My paper investigates this for

$$
\mathcal{E} (E)=P_\alpha(E)+\mathcal{G}(E).
$$

Observe that two main ingredients define the nonlocal free energy of a set  $E \subset
\mathbb{R}^n$: $P_\alpha(E)$; and, $\mathcal{G}(E)=\int_E g(x)dx$.

In this context, the nonlocal Almgren problem is to investigate the convexity of a minimizer $E_m$ under the assumption that $g$ is convex, refer to \cite[p. 146]{MR2807136} to understand the local Almgren problem. 
Several theorems may be shown without convexity on $g$. In particular, a complete theory begins via $g \in L_{loc}^\infty$  (in several contexts, one can assume $g \in L_{loc}^1$).\\

\noindent 1. Assuming coercivity, there are minima for $m>0$ \cite{MR3640534}.\\
2. Assuming that $m$ is sufficiently small, all minimizers are convex \cite{MR4674821}.\\
3. One may construct a $g$ which is convex so that there are no minimizers if $m>0$,  see Theorem \ref{ghp}. \\
4. Assuming $g(x)=h(|x|)$, $h:
\mathbb{R}^+
\rightarrow \mathbb{R}^+$ is increasing, convex, $h(0)=0$, there exists a stability estimate similar to \eqref{es}, see Theorem \ref{@'}.\\
5. Assuming $g \in L_{loc}^\infty$ and up to sets of measure zero
$g$ admits unique
minimizers $E_m$, there exist energy moduli. In addition, assuming that $m$ is sufficiently small, the energy modulus has a product structure, see Proposition \ref{Ko} and Theorem \ref{@'z}.\\
 6. Upper bounds for the moduli are obtained with minimal assumptions, see Theorem \ref{@7z'}\\

In the paper, the novelty mostly is in 4. Interestingly, 3, 5, 6 can be obtained, via minor changes, as in \cite{MR4730410, qk} (in 6, one utilizes \cite{MR4674821}) and hence these proofs are in the appendix. 
 \section{Stability for nonlocal free energy minimization}
 
\begin{thm} \label{@'}
Suppose $g(x)=h(|x|)$, $h:
\mathbb{R}^+
\rightarrow \mathbb{R}^+$ is increasing, convex, $h(0)=0$. Let $m>0$, $|
B_a|=|E|=m$,
then\\ 
\noindent (i)
$$
\mathcal{E}(E) - \mathcal{E}(B_a)  \ge r(m,  n,\alpha) ||\chi_E-\chi_{B_a}||_{L^1}^4
$$
for some $r(m,n,\alpha)> 0$; \\

\noindent (ii) supposing $\hat{a}>a$, $E \subset B_{\hat{a}}$, then 

$$
\mathcal{E}(E) - \mathcal{E}(B_a)  \ge r(m, \hat{a}, n,\alpha) ||\chi_E-\chi_{B_a}||_{L^1}^2
$$
for some explicit $r(m, \hat{a}, n,\alpha)> 0$. 
\end{thm} 
\begin{rem}
The exponent 4 is appearing when considering the stability of the isoperimetric inequality. Hall proved a stability theorem with 4 and conjectured that the exponent can be replaced with 2 \cite{zbMATH00055774}. The conjecture was proven in \cite{MR2456887, MR2672283}.
\end{rem}
\begin{proof}

\noindent  (i)\\
Assume 
$$
\mathcal{E}(E) - \mathcal{E}(B_a) \ge \nu,
$$
for $\nu>0$. Then since

$$
|E \Delta B_a|^4 \le 16m^4,
$$
\begin{align*} 
|E \Delta B_a|^4 &\le 16m^4\\
&\le \frac{16m^4}{\nu} \nu \le \frac{16m^4}{\nu}(\mathcal{E}(E) - \mathcal{E}(B_a)).
\end{align*}
Therefore it is sufficient to prove: there exists $\nu>0$ so that if 
$$
\mathcal{E}(E) - \mathcal{E}(B_a) \le \nu,
$$
then 
\begin{equation} \label{w}
w(|E \Delta B_a|) \le \mathcal{E}(E) - \mathcal{E}(B_a),
\end{equation}
where $w(t)=r(m,n,\alpha)t^4$. To start, the existence of a modulus $w$ is proved so that \eqref{w} is true (observe this also follows from a compactness proof, cf. Proposition \ref{Ko}, as soon as one obtains that the ball is the unique minimizer; the new argument has advantages in the context of explicitly encoding estimates to identify $w(t)=r(m,n,\alpha)t^4$).
Suppose
$$T: E\setminus B_a \rightarrow B_a \setminus E$$
denotes the Brenier map between $\mu_* = \chi_{E\setminus B_a} dx$ and
$
\nu_*=\chi_{B_a \setminus E} dx$ \cite{pFZ4}. Since
$$
T_{\#} \mu_*=\nu_*,
$$
one has
\begin{equation} \label{t}
\int_{E \setminus B_a} h(|T(x)|)dx=\int_{B_a \setminus E} h(|x|)dx;
\end{equation}
thus, \eqref{t} and monotonicity of $h$ yield

\begin{align}  \label{sh}
\int_{B_a} h dx&=\int_{E \cap B_a} h dx+\int_{B_a \setminus E} h(|x|)dx \notag\\
&=\int_{E \cap B_a} h dx +\int_{E \setminus B_a} h(|T(x)|)dx \notag\\
&\le \int_{E \cap B_a} h dx + \int_{E \setminus B_a} h(|x|)dx \notag\\
&=\int_E h dx.
\end{align}

Observe that via $|T(x)| \le a$ and the above,
$$
\int_{E \setminus B_a} [h(|x|)-
h(a)]dx \le \int_E h dx-\int_{B_a} h dx. 
$$
Hence the previous inequality and \cite{MR3322379} imply

\begin{align*}
\mathcal{E}(E) - \mathcal{E}(B_a)&= P_\alpha(E)-P_\alpha(B_a)+ \mathcal{G}(E)-\mathcal{G}(B_a)\\
& \ge P_\alpha(B_a)\frac{A^2(E)}{C(n,\alpha)}+\int_E h dx-\int_{B_a} h dx\\
& \ge P_\alpha(B_a)\frac{A^2(E)}{C(n,\alpha)}+ \int_{E \setminus B_a} [h(|x|)-h(a)]dx,
\end{align*}
$$
A(E)=\inf\{\frac{|E \Delta (B_a+x)|}{|E|}: x \in \mathbb{R}^n\}.
$$
Let $z_E$ achieve 
$$
A(E)=\frac{|E \Delta (B_a+z_E)|}{|E|}.
$$
Hence
\begin{equation} \label{qu}
\Big(\frac{|E \Delta (B_a+z_E)|}{|E|}\Big)^2 \le \tilde a_1\Big(\mathcal{E}(E) - \mathcal{E}(B_a)\Big),
\end{equation}
$\tilde a_1=\tilde a_1(n,\alpha,a)>0.$
Assume
$$
 \mathcal{E}(E_i) - \mathcal{E}(B_a) \rightarrow 0,
$$
$|E_i|=m=|B_a|$. Observe 
$$
|E_i \Delta (B_a+z_{E_i})| \rightarrow 0.
$$
Let $A(i)\subset B_a+z_{E_i}$, $|A(i)|\ge p>0$.
Thus 
$$
\int |\chi_{E_i}(x)-\chi_{B_a+z_{E_i}}(x)|dx \rightarrow 0
$$
yields
$$
\int_{A(i)} |\chi_{E_i}(x)-1|dx \rightarrow 0.
$$
By the triangle inequality,
\begin{align*}
||A(i)\cap E_i|-|A(i)||&=|\int_{A(i)} (\chi_{E_i}(x)-1)dx|\\
&\le \int_{A(i)} |\chi_{E_i}(x)-1|dx \rightarrow 0.
\end{align*}
This thus implies
$$
|A(i)\cap E_i| \ge \frac{p}{7}
$$
assuming $i$ is large.\\

\noindent Claim 1: $\sup_{i} |z_{E_{i}}| <\infty$.\\

\noindent Proof of Claim 1:\\
Assume not. Then modulo a subsequence,
$$
|z_{E_{i}}| \rightarrow \infty.
$$
Since $A(i)\subset B_a+z_{E_i}$, one obtains
$$
\inf_{A(i)\cap E_i} g \rightarrow \infty
$$
thanks to the (strict) monotonicity of $g$. Thus
\begin{align*}
\infty &=\lim_i \frac{p}{7} \inf_{A(i)\cap E_i} g \\
&\le  \int_{A(i)\cap E_i} g \\
 &\le \int_{E_i} g \rightarrow \int_B g <\infty,
\end{align*}
and this contradiction yields Claim 1.\\

\noindent Claim 2: $ |z_{E_i}| \rightarrow 0$.\\

\noindent Proof of Claim 2:\\
If one can find a subsequence (continued to have the same index) so that 
$\inf_{i} |z_{E_{i}}| > 0$,
observe via the positive bound that up to possibly another subsequence,

$$
z_{E_{i}} \rightarrow z \neq 0.
$$
In particular, 
\begin{align*}
|E_i \Delta (B_a+z)|&\le |(B_a+z_{E_i})\Delta (B_a+z)|+|E_i \Delta (B_a+z_{E_i})| \\
&\le a_*|z_{E_i}-z|+|B_a|\sqrt{\tilde a_1(\mathcal{E}(E_i) - \mathcal{E}(B_a))}\\
&\rightarrow 0.
\end{align*}
Hence
$$
\chi_{E_{i}} \rightarrow \chi_{B_a+z} \hskip .4in \text{in $L^1$}.
$$
and this readily yields, mod a subsequence,
$$
\chi_{E_{i}} \rightarrow \chi_{B_a+z} \hskip .4in \text{a.e.}.
$$
Therefore via Fatou
\begin{align*}
\int_{B_a+z} g \le \liminf_i \int_{\mathbb{R}^n} g \chi_{E_{i}}=\mathcal{G}(B_a),
\end{align*}
Define $H=B_a+z$.
Consider $\{v_1, v_2, \ldots \}$ directions which generate via
Steiner
symmetrization with respect to the planes through the origin and
with normal $v_i$,
$H_{v_1}, H_{v_2}, \ldots$ and
$$
H_{v_i} \rightarrow B_a.
$$
Set
$$
H_1=\{(x_2, \ldots, x_n): (x_1, x_2, \ldots, x_n) \in H\}
$$
$$H^{x_2, x_3, \ldots, x_n}=\{x_1: (x_1, x_2, \ldots, x_n) \in H\}.$$
Note that one may let $\frac{x_1}{|x_1|}=v_1$; hence the monotonicity and complete symmetry of $g$  imply

\begin{align*}
\int_{H_{v_1}} g dx&=\int_{H_1} \Big( \int_{-\frac{|H^{x_2, x_3, \ldots, x_n}|}{2}}
^{\frac{|H^{x_2, x_3,
\ldots, x_n}|}{2}} g dx_1\Big) dx_2\ldots dx_n \\
&\le \int_{H_1} \Big( \int_{H^{x_2, x_3, \ldots, x_n}} g
dx_1\Big) dx_2\ldots
dx_n \\
&=\int_H g dx.
\end{align*}
Moreover, via the radial property of $g$, one may rotate the coordinate $x_1$
so that $
\frac{x_1}{|x_1|}=v_2$ and iterate the argument above:
$$
\int_{H_{v_1}} g dx \ge \int_{H_{v_2}} g dx;
$$
in particular, note that via $H\neq B_a$ \& the monotonicity of $g$ (the strict monotonicity), there is one direction so that the inequality is strict, therefore
$$
\int_{H} g dx > \int_{B_a} g dx.
$$
Hence one obtains a contradiction
\begin{align*}
 \int_{B_a} g dx&<\int_{H} g=\int_{B_a+z} g\\
 & \le \liminf_i \int_{\mathbb{R}^n} g \chi_{E_{i_l}}\\
 &=\mathcal{G}(B_a)= \int_{B_a} g dx.
\end{align*}

Note that this proves:\\

assuming
$$
 \mathcal{E}(E_i) - \mathcal{E}(B_a) \rightarrow 0,
$$
it follows that

$$
|E_i \Delta (B_a+z_{E_i})| \rightarrow 0,
$$

\begin{equation} \label{x}
z_{E_i} \rightarrow 0.
\end{equation}
Now note
\begin{align*}
|E_i \Delta B_a| &\le |E_i \Delta (B_a+z_{E_i})|+|B_a \Delta (B_a+z_{E_i})| \\
&\le  |E_i \Delta (B_a+z_{E_i})|+a_*|z_{E_i}|\rightarrow 0.
\end{align*}
Hence there is some modulus $w$ so that
$$
\mathcal{E}(E) - \mathcal{E}(B_a)  \ge w(|E\Delta B_a|).
$$

Supposing 
$\hat{a_1}>a_1>0$, $E_* \subset B_{\hat{a_1}}$, $|E_*|=|B_{a_1}|$, then via strict monotonicity and convexity of $h$, one obtains that the subdifferential
$$
\partial_+h(a_1) \neq \emptyset
$$
is compact and 
\begin{equation} \label{gr}
\inf_{\partial_+h(a_1)} |x|>0:
\end{equation}
assume $0\in \partial_+h(a_1)$, one then can use $a_1>0$ and convexity to deduce that $g$ has a global minimum at $a_1$ and this is a contradiction via the strict monotonicity ($g(0)=0$, $g \ge 0$), therefore this shows \eqref{gr}. 
Hence thanks to \cite{qk},
there exists 
\begin{equation} \label{ge}
r_{\hat{a_1},a_1, \partial_+h(a_1)}^*=\Big(\frac{1}{ \inf_{\partial_+h(a_1)} |x|A_*}\Big)^{\frac{1}{2}}>0,
\end{equation}
$A_*=A_*(\hat{a_1}, n,a)>0$, so that 
\begin{equation} \label{st}
r_{\hat{a_1},a_1, \partial_+h(a_1)}^*\Big[\mathcal{G}(E_*)-\mathcal{G}(B_{a_1})\Big]^{\frac{1}{2}} \ge |E_*\Delta B_{a_1}|.
\end{equation}

Supposing $|E|=|B_a|$,
via the previous argument (cf. \eqref{x}) if 
$$
 \mathcal{E}(E) - \mathcal{E}(B_a) \rightarrow 0,
$$
it follows that

$$
z_{E} \rightarrow 0.
$$
Hence if 

$$
\mathcal{E}(E) - \mathcal{E}(B_a)\le \nu
$$
where
$\nu$ small,
then $|z_{E} |$ is small.

Next, 
$$
(B_a+z_E) \subset B_{a+|z_E|}
$$
yields
$$
E\setminus B_{a+|z_E|} \subset E\setminus(B_a+z_E)
$$
and thus utilizing \eqref{qu}
\begin{equation} \label{sq}
|E\setminus B_{a+|z_E|}| \le | E\setminus(B_a+z_E)| \le \sqrt{|B_a|^2\tilde a_1(\mathcal{E}(E) - \mathcal{E}(B_a))}.
\end{equation}
Set 
$$
E_*=E \cap (B_{a+|z_E|});
$$
then note
$$
E=E_* \cup (E \setminus (B_{a+|z_E|}))
$$
\& thanks to \eqref{sq}
\begin{align*}
m&=|E_*|+|E \setminus (B_{a+|z_E|})|\\
&\le |E_*|+\sqrt{|B_a|^2\tilde a_1(\mathcal{E}(E) - \mathcal{E}(B_a))}.
\end{align*}
Therefore
\begin{equation} \label{th}
m-|E_*|\le \sqrt{|B_a|^2\tilde a_1(\mathcal{E}(E) - \mathcal{E}(B_a))};
\end{equation}
next, consider $a_1>0$ via
$$
|E_*|=|B_{a_1}|=a_1^n|B_1|.
$$
Observe that \eqref{th} easily implies
$$
m-\sqrt{|B_a|^2\tilde a_1(\mathcal{E}(E) - \mathcal{E}(B_a))}\le |E_*|=a_1^n|B_1|,
$$
hence
$$
\Big(\frac{m-\sqrt{|B_a|^2\tilde a_1(\mathcal{E}(E) - \mathcal{E}(B_a))}}{|B_1|}\Big)^{\frac{1}{n}} \le a_1.
$$
In particular, since the energy difference is small $\mathcal{E}(E) - \mathcal{E}(B_a)\le \nu$, there exists a lower bound on $a_1$ via $a, n, \alpha$.
Thanks to $|z_{E} |$ being small and \eqref{st}, one may choose
$$
5a>\hat{a_1}>a+|z_E|
$$
such that 
\begin{equation} \label{bdd}
r_{\hat{a_1},a_1, \partial_+h(a_1)}^*\Big[\mathcal{G}(E_*)-\mathcal{G}(B_{a_1})\Big]^{\frac{1}{2}} \ge |E_*\Delta B_{a_1}|,
\end{equation}
where $r_{\hat{a_1},a_1, \partial_+h(a_1)}^*$ is bounded via $a, n, \alpha$.
Now since $a\ge a_1$, \eqref{th} implies
\begin{align*}
\mathcal{G}(E_*)-\mathcal{G}(B_{a_1})& =\mathcal{G}(E_*)-\mathcal{G}(B_{a})+\mathcal{G}(B_{a})-\mathcal{G}(B_{a_1})\\
&\le \mathcal{G}(E_*)-\mathcal{G}(B_{a})+\big(\sup_{B_a} g \big)|B_a\setminus B_{a_1}| \\
& \le \mathcal{G}(E_*)-\mathcal{G}(B_{a})+\big(\sup_{B_a} g \big)(|B_a|- |B_{a_1}|)\\
&=\mathcal{G}(E_*)-\mathcal{G}(B_{a})+\big(\sup_{B_a} g\big) (m- |E_*|)\\
&\le \mathcal{E}(E) - \mathcal{E}(B_a)+\big(\sup_{B_a} g\big) \sqrt{|B_a|^2\tilde a_1(\mathcal{E}(E) - \mathcal{E}(B_a))}.
\end{align*}
 In particular, \eqref{bdd} and the above inequality imply
\begin{align} \label{e}
|E_*\Delta B_{a_1}|& \le r_{\hat{a_1},a_1, \partial_+h(a_1)}^*\Big[\mathcal{G}(E_*)-\mathcal{G}(B_{a_1})\Big]^{\frac{1}{2}} \notag\\
& \le r_{\hat{a_1},a_1, \partial_+h(a_1)}^*\Big[ \mathcal{E}(E) - \mathcal{E}(B_a)+\sup_{B_a} g \sqrt{|B_a|^2\tilde a_1(\mathcal{E}(E) - \mathcal{E}(B_a))}\Big]^{\frac{1}{2}}.
\end{align}
Also, \eqref{th} \& \eqref{sq} yield
\begin{equation} \label{ba}
|B_a\Delta B_{a_1}|= |B_a\setminus B_{a_1}| \le \sqrt{|B_a|^2\tilde a_1(\mathcal{E}(E) - \mathcal{E}(B_a))}
\end{equation}
\begin{equation} \label{bd}
|E_* \Delta E| =|E \setminus (B_{a+|z_E|})| \le \sqrt{|B_a|^2\tilde a_1(\mathcal{E}(E) - \mathcal{E}(B_a))}.
\end{equation}
Hence \eqref{ba}, \eqref{e}, \eqref{bd}, \eqref{sq}, \& the triangle inequality in $L^1$ imply
\begin{align*}
 |B_a \Delta (B_{a}+z_E)| &\le  |B_a \Delta B_{a_1}| +|B_{a_1} \Delta E_*|+|E \Delta E_*|+|E\Delta (B_{a}+z_E)|\\
&\le \alpha_* \Big[\mathcal{E}(E)-\mathcal{E}(B_{a})\Big]^{\frac{1}{4}}.
\end{align*}
Last,
\begin{align*}
|E\Delta B_a|& \le |E\Delta (B_a+z_E)| +|(B_a+z_E)\Delta B_a|\\
&\le \sqrt{|B_a|^2\tilde a_1(\mathcal{E}(E) - \mathcal{E}(B_a))}+ \alpha_* \Big[\mathcal{E}(E)-\mathcal{E}(B_{a})\Big]^{\frac{1}{4}}\\
&\le \overline{\alpha}_* \Big[\mathcal{E}(E)-\mathcal{E}(B_{a})\Big]^{\frac{1}{4}},
\end{align*}
$\overline{\alpha}_*=\overline{\alpha}_*(a, n, \alpha)>0$.

\noindent (ii) \\
\noindent Supposing  $\hat{a}>a$, $E \subset B_{\hat{a}}$, then one can show similarly to the proof in (i) (cf. \eqref{st}, \eqref{ge}) that there exists some constant $r_{m, \hat{a},n, \alpha}^*>0$ (that is explicit) so that 
$$
r_{m, \hat{a},n, \alpha}^*\Big[\mathcal{E}(E)-\mathcal{E}(B_a)\Big]^{\frac{1}{2}} \ge r_{m, \hat{a},n, \alpha}^*\Big[\mathcal{G}(E)-\mathcal{G}(B_a)\Big]^{\frac{1}{2}} \ge |E\Delta B_a|.
$$

\end{proof}
\begin{rem}
Assuming $E \subset B_{\hat{a}}$, the estimate is optimal: set $g(y)=|y|^2$, 
\begin{align*}
\mathcal{E}(B_a+x)-\mathcal{E}(B_a)&=\int_{B_a} (2\langle y,x\rangle +|x|^2 )dy =|x|^2|B_a|;
\end{align*}
supposing $|x|$ is small, since
$$
|(B_a+x) \Delta B_a| \approx |x|,
$$
 
$$
|(B_a+x) \Delta B_a|^2 \approx |x|^2.
$$
\end{rem}

\begin{rem} \label{hv}
 In the argument  of the theorem, \eqref{sh} implies that balls minimize the energy also when $g$ is non-decreasing, radial, and possibly non-convex. 
 \end{rem}
\begin{rem}
In the theorem,  a quantitative
inequality is proven without
translation invariance. In particular, the invariance class is completely identified and stability is not modulo translations like the quantitative anisotropic
isoperimetric inequality. Supposing a context where the set is in a convex cone \cite{MR3023863, I20, MR1897393g, MR4380032}, translations are
crucial: assuming the cone contains no line, the quantitative
term is without
translations. 
\end{rem}

\begin{cor} \label{@4'}
Suppose $g(x)=h(|x|)$, where $h$ is non-negative, non-decreasing, not identically zero, and homogeneous of
degree $\nu$. Let $m>0$ and assume $E_m$ is the minimizer with $|
E_m|=m$, set
$$
m_{\alpha,\nu, g}=|B_1|\Big[\frac{\alpha(n-\alpha)}{\nu(n+\nu)} \frac{P_\alpha(B_1)}{\int_{B_1}g(x)dx}\Big]^{\frac{n}{\nu+\alpha}}
$$
it then follows that
$$
m \mapsto \mathcal{E}(E_m)
$$
is concave on $(0,m_{\alpha,\nu, g})$ and convex on $(m_{\alpha,\nu, g}, \infty)$.
\end{cor}
\begin{proof}
Thanks to Remark \ref{hv}, 
$E_m=B_r$, $r=\Big[ \frac{m}{|B_1|}\Big]^{\frac{1}{n}}$,
which therefore
implies
$$
\mathcal{E}(E_m) =P_\alpha(B_1)(\frac{1}{|B_1|
^{\frac{n-\alpha}
{n}}})m^{\frac{n-\alpha}{n}}+(\frac{1}{|B_1|^{\frac{\nu+n}{n}}})
(\int_{B_1} h(|
x|)dx)m^{\frac{n+\nu}{n}}.
$$
The critical mass $m_{\alpha,\nu, g}$ therefore  is calculated  with the
second derivative.
\end{proof}

\section{Appendix}
\subsection{Compactness: existence of moduli}
\begin{prop} \label{Ko}
If $m>0$, $g \in L_{loc}^\infty$, and up to sets of measure zero,
$g$ admits unique
minimizers $E_m$, then for $\epsilon>0$ there exists
$w_m(\epsilon)>0$ such that if
$|E|=|E_m|$, $E \subset B_R$, $R=R(m)$, and
$$
|\mathcal{E}(E)-\mathcal{E}(E_m)|<w_m(\epsilon),
$$
then
$$
\frac{||\chi_E-\chi_{E_m}||_{L^1}}{|E_m|}<\epsilon.
$$  
\end{prop}
\begin{proof}
Assume the assumptions do not yield the conclusion, then there exists $\epsilon>0$ and for $a>0$, there
exist $E_a' \subset
B_R$ and $E_m$, $|E_a'|=|E_m|=m$,
$$
|\mathcal{E}(E_a')-\mathcal{E}(E_m)|<a,
$$
$$
\frac{|E_m\Delta E_a'|}{|E_m|} \ge \epsilon.
$$
Define $a=\frac{1}{k}$, $k \in \mathbb{N}$; therefore 
there exist
$E_{\frac{1}{k}}'$, $|E_{\frac{1}{k}}'|=m$,
$$
|\mathcal{E}(E_m)-\mathcal{E}(E_{\frac{1}{k}}')|\le \frac{1}{k},
$$
$$
\frac{|E_m\Delta E_{\frac{1}{k}}'|}{|E_m|} \ge \epsilon.
$$
In particular

\begin{align*}
P_\alpha(E_{\frac{1}{k}}')& \le \mathcal{E}(E_{\frac{1}{k}}')\\
& \le \frac{1}{k}+\mathcal{E}(E_m),
\end{align*}
$$E_{\frac{1}{k}}' \subset B_R,$$
and via the compactness $H^{\frac{\alpha}{2}} \hookrightarrow L_{loc}^1$, up to a
subsequence,
$$
E_{\frac{1}{k}}' \rightarrow E' \hskip .3in in \hskip .08in
L^1(B_R);
$$
thus $|E'|=m$ because of the triangle inequality in
$L^1(B_R)$. Next
$$
\mathcal{E}(E') \le \liminf_k \mathcal{E}(E_{\frac{1}{k}}')
=\mathcal{E}(E_m)
$$
implies
$E'$ is a minimizer, contradicting
$$
\frac{|E_m\Delta E'|}{|E_m|} \ge \epsilon
$$
via the uniqueness.
\end{proof}

\subsection{Bounds on the moduli}
Identifying the modulus $w_m(\epsilon)$ is in general complex. Additional conditions illuminate interesting properties as illustrated in the first theorem. In the subsequent theorem, upper bounds are illustrated with minimal assumptions. 
\begin{thm} \label{@7z'}
Suppose $m>0$, and up to sets of measure zero, $g$ admits unique
minimizers
$E_m \subset B_{r(m)}$ in the collection of sets of finite perimeter,
then\\
\noindent (i) if $g \in W_{loc}
^{1,1}$ is locally uniformly differentiable, there exists $\alpha_1=\alpha_1(E_m)>0$ such that if $
\epsilon< \alpha_1$, one has
$$
w_m(\epsilon) \le \lambda_{g, E_m}(m, \epsilon)=o_{g, E_m}(\frac{\epsilon m}
{a_{m_*}}),
$$
$$
a_{m_*}=\alpha_*|D\boldsymbol{\chi}_{E_m} \cdot w_*|(\mathbb{R}^n)
$$
with $\alpha_*>0$, $w_* \in \mathbb{R}^n$;\\
\noindent (ii)
if $g \in W_{loc}
^{1,\infty}$, there exists $m_a>0$ such that if $m<m_a$, $
\epsilon< \alpha_1$, then
$$\lambda_{g, E_m}(m, \epsilon)= \alpha_{1_*} m^{\frac{1}{n}}o_{g, E_m}(\epsilon),$$
$$o_{g, E_m}(\epsilon) =\epsilon \int_{E_m} o_{x,g}(1) dx,$$
$$
\int_{E_m} o_{x,g}(1) dx \rightarrow 0
$$
as $\epsilon \rightarrow 0^+$,  $\alpha_{1_*}>0$;\\
\noindent (iii) if $g \in C^1$,  there exist $m_a, \alpha_m>0$ such that if 
$m<m_a$, $\epsilon<
\alpha_m$,  then
$$\lambda_{g, E_m}(m, \epsilon)= \alpha_{2_*} m^{1+\frac{1}{n}}o_{g}(\epsilon),$$
$\alpha_{2_*}=\alpha_{2_*}(||D g||_{L^\infty(B_{r(m_a)})})>0$;\\
\noindent (iv)  if $g \in W_{ loc}
^{2,1}$ is twice
differentiable, there exists $\alpha_m>0$ such that if $\epsilon< \alpha_m$,
$$
w_m(\epsilon) \le a_m \epsilon^2,
$$
$$a_m=\Big(||D^2 g||_{L^1(E_m)} +m\frac{1}{6} \Big) (\frac{m}
{\alpha_*|D
\boldsymbol{\chi}_{E_m} \cdot w_*|(\mathbb{R}^n)} )^2$$
with $\alpha_*>0$;\\
(v) if  $g
\in W_{loc}^{2, \infty}$, there exists $m_a>0$ such that if
$m<m_a$, then
$$
w_m(\epsilon) \le a_m \epsilon^2,
$$
$$
a_m=a_*m^{1+\frac{2}{n}},
$$
$a_*=a_*(||D^2 g||_{L^\infty(B_{r(m_a)})})>0$.
\end{thm}

\begin{proof}
Assuming $m$ is small, the convexity of $E_m$ is established in \cite{MR4674821} and utilized for $(ii), (iii), (v)$ similar to the proof in \cite{qk}. The argument for $(i)$ \& $(iv)$ is the same as in \cite{qk}.
\end{proof}

\begin{rem}
The existence of bounded $E_m$ was proven in \cite{MR3640534} via assuming coercivity of $g$. Also, a minimizer $E_m$ is up to a closed set with Hausdorff dimension at most $n-3$, a $C^{2,a} $ set, in particular, much more smooth than just a set of finite perimeter. One may in some contexts preclude translations (e.g. supposing that $g$ is strictly convex). If $g$ is zero on some small ball, then note that if the mass is very small, uniqueness is only up to translations and sets of measure zero.
\end{rem}

\begin{rem}
The assumption $g \in W_{loc}
^{1,1}$ is locally uniformly differentiable in (i) can be weakened.
\end{rem}

\begin{rem} \label{zo8}
Theorem \ref{@'} encodes the modulus in an explicit way. In
particular,
supposing $m$ is small, if $g(|
x|)=|x|^2$,
a possible modulus is $w_m(\epsilon)=a_1\epsilon^2 m^3$. Thus the $\epsilon$ dependence in (v) is
optimal.
Observe also that via the quadratic, the $m$ dependence is almost sharp in two dimensions. The optimal $m$--modulus also encodes $a_*(||D^2 g||_{L^\infty(B_{r(m_a)})})$.
\end{rem} 

\subsection{Non-existence of minimizers}
In the general case when $g$ is convex, one does not have existence.  

\begin{thm} \label{ghp}
There exists $g \ge 0$ convex such that $g(0)=0$ and such that if $m>0$, then there is no solution to
$$
\inf\{\mathcal{E}(E): |E|=m\}.
$$ 
\end{thm}
\begin{proof}
The definition of $g$ via the construction in \cite{MR4730410} also works in this case:
\begin{equation*}  
g(x,y)=\begin{cases}
x^2(1-y)+x^2y^2 & \text{if } y\le0\\
\frac{x^2}{1+y}  & \text{if } y > 0.
\end{cases}
\end{equation*}
Set  $e_2=(0,1)$, $a>0$; the potential is non-increasing in the $y$-variable and strictly decreasing if $x\neq 0.$
in particular, note  if a minimizer $E_m$ exists,
$$
\int_{E_m+ae_2} g(x,y) dxdy < \int_{E_m} g(x,y) dxdy, 
$$
therefore via the translation invariance of $P_\alpha$,
$$
\mathcal{E}(E_m+ae_2)< \mathcal{E}(E_m),
$$
a contradiction.

\end{proof}

In particular, coercivity or another condition is necessary.

\subsection{A product property} 
If $\alpha, g$ are given, an invariance map  of the nonlocal free energy is a transformation 
\begin{align*}
&A \in \mathcal{A}_m=\mathcal{A}_{\alpha,g,m}\\
&=\{A: Ax=A_ax+x_a, x_a \in \mathbb{R}^n, \mathcal{E}(A_aE)=\mathcal{E}(E), |A_aE|=|E|=m \text{ for some minimizer $E$}  \}.
\end{align*}
Uniqueness of minimizers can only be true up to $P_\alpha$ sets of measure zero and an invariance map. Note that in many classes of potentials, assuming $m$ is small, $A \in \mathcal{A}_m$ is a translation; a simple case is: assume $g$ is zero on a ball $B$, therefore if $m$ is small, $A_a=I_{n \times n}$, $x_a \in \mathbb{R}^n$ is such that $B_a+x_a \subset \{g=0\}$ when $B_a \subset B$.  

Supposing $g$ is locally bounded, the stability modulus, in the context of small mass, is a product in any dimension.

\begin{thm} \label{@'z}
Suppose $g \in L_{loc}^\infty(\{g<\infty\})$ admits minimizers $E_m \subset B_{R}$ for all $m$ small. There exists $m_0>0$ and a modulus of continuity $q(0^+)=0$ such that for all $m<m_0$ there exists $\epsilon_0>0$ such that for all $0<\epsilon<\epsilon_0$ and for all minimizers $E_m \subset B_R$, $E \subset B_R$, $|E|=|E_m|=m<m_0$, if 
$$
|\mathcal{E}(E_m)-\mathcal{E}(E)| < a(m,\epsilon, \alpha)=q(\epsilon)m^{\frac{n-\alpha}{n}},
$$
there exists an invariance map $A \in \mathcal{A}_m$  such that  
$$
\frac{||\chi_E-\chi_{AE_m}||_{L^1}}{|E_m|} < \epsilon.
$$
Also, $AE_m \approx E_m + \alpha_m$, $ \alpha_m \in \mathbb{R}^n$ cf. \eqref{appr}.
\end{thm}
\begin{rem}
If $g(x)=h(|x|)$, $h:
\mathbb{R}^+\rightarrow \mathbb{R}^+$ is increasing, $h(0)=0$, then $A$ is the identity in the conclusion of Theorem \ref{@'z}: if $m>0$, then
$$\mathcal{A}_{\alpha, h(|x|),m}=\{I_{n \times n}\}.$$
\end{rem}
\begin{proof}
Via contradiction, suppose the theorem is not true, then for $m_0>0$, for all moduli $q$, there exists $m<m_0$ such that 
 for $\epsilon_0 \in (0,2]$ there exists $\epsilon<\epsilon_0$ and $E_{m, \epsilon_0}, E_{m, \epsilon_0}' \subset B_R$, $|E_{m, \epsilon_0}|=|E_{m, \epsilon_0}'|=m$,

$$
|\mathcal{E}(E_m)-\mathcal{E}(E_{m}')| < q(\epsilon)m^{\frac{n-\alpha}{n}},
$$ 
and

\begin{equation} \label{wl4}
\inf_{A \in \mathcal{A}_{m}} \frac{|E_{m,\epsilon_0}' \Delta AE_{m,\epsilon_0}|}{|E_{m,\epsilon_0}|} \ge \epsilon>0.
\end{equation}
In particular, let $m_0=\frac{1}{k}$, $w_k \rightarrow 0^+$, $\hat{q}$ define a modulus of continuity, and let

\begin{equation} \label{t3}
q_k=w_k \hat{q}(\epsilon).
\end{equation}
Therefore there exists $m_k<\frac{1}{k}$ such that for a fixed $\epsilon_0 \in (0,2]$ there exists $\epsilon<\epsilon_0$ and $E_{m_k, \epsilon_0}, E_{m_k, \epsilon_0}' \subset B_R$, $|E_{m_k, \epsilon_0}|=|E_{m_k, \epsilon_0}'|=m_k<\frac{1}{k}$,
$$
|\mathcal{E}(E_{m_k})-\mathcal{E}(E_{m_k}')| < q_k m_k^{\frac{n-\alpha}{n}},
$$ 
and
\begin{equation} \label{wlje1}
\inf_{A \in \mathcal{A}_{m_k}}  \frac{|E_{m_k,\epsilon_0}'\Delta AE_{m_k,\epsilon_0}|}{|E_{m_k,\epsilon_0}|} \ge \epsilon>0.
\end{equation}
Let
\begin{equation} \label{a_k}
a_k=q_k m_k^{\frac{n-\alpha}{n}},
\end{equation} 

$E_{m_k}=E_{m_k,\epsilon_0}$, $E_{m_k}'=E_{m_k,\epsilon_0}'$. Define $\gamma_k=(\frac{|B_1|}{m_k})^{\frac{1}{n}}$,
$$|\gamma_k E_{m_k}|=|B_1|.$$  
Observe via the nonlocal isoperimetric inequality and the minimality, 

\begin{align*}
P_\alpha(\frac{1}{\gamma_k}B_1)+\int_{E_{m_k}}g &\le P_\alpha(E_{m_k})+\int_{E_{m_k}}g \\
&\le P_\alpha(\frac{1}{\gamma_k}B_1)+\int_{\frac{1}{\gamma_k}B_1}g 
\end{align*}

$$
P_\alpha(E_{m_k})-P_\alpha(\frac{1}{\gamma_k}B_1) \le (\sup_{\frac{1}{\gamma_k}B_1} g )m_k
$$

$$
\gamma_k^{n-\alpha}\Big(P_\alpha(E_{m_k})-P_\alpha(\frac{1}{\gamma_k}B_1) \Big)\le \gamma_k^{n-\alpha}\Big((\sup_{\frac{1}{\gamma_k}B_1} g )m_k\Big)
$$

$$
\Big(P_\alpha(\gamma_kE_{m_k})-P_\alpha(B_1) \Big)\le \Big((\sup_{\frac{1}{\gamma_k}B_1} g )(\gamma_k^{n-\alpha} m_k)\Big)
$$

$$
\Big(P_\alpha(\gamma_kE_{m_k})-P_\alpha(B_1) \Big)\le \Big((\sup_{(\frac{1}{\gamma_k})B_1} g )(\frac{|B_1|}{\gamma_k^\alpha})\Big).
$$

Since $m_k \rightarrow 0$ as $k \rightarrow \infty$, it thus follows that $\gamma_k \rightarrow \infty$, hence

\begin{align*}
\delta(\gamma_k E_{m_k})&=\frac{P_\alpha(\gamma_k E_{m_k})}{P_\alpha(B_1)} - 1 \rightarrow 0.
\end{align*}

Next, by the triangle inequality,

\begin{align*}
|P_\alpha(E_{m_k}')&-P_\alpha(E_{m_k})|\\
&=|[\mathcal{E}(E_{m_k}')-\mathcal{E}(E_{m_k})]+[\int_{E_{m_k}} g(x)dx-\int_{E_{m_k}'} g(x)dx]|\\
&\le |\mathcal{E}(E_{m_k}')-\mathcal{E}(E_{m_k})|+\int g(x)|\chi_{E_{m_k}'}-\chi_{E_{m_k}}| dx\\
&<a_k+\int_{E_{m_k}' \Delta E_{m_k}} g(x)dx.
\end{align*}

In particular,  

\begin{align*}
|P_\alpha(\gamma_kE_{m_k}')&-P_\alpha(\gamma_kE_{m_k})|\\
&<\gamma_k^{n-\alpha}a_k+2\gamma_k^{n-\alpha} m_k(\sup_{B_R \cap \{g<\infty\}} g) \\
&=|B_1|^{\frac{n-\alpha}{n}}\frac{a_k}{m_k^{\frac{n-\alpha}{n}}}+2|B_1|^{\frac{n-\alpha}{n}}(\sup_{B_R \cap \{g<\infty\}} g) m_k^{1-\frac{n-\alpha}{n}}
\end{align*}
and since from \eqref{t3} and \eqref{a_k}, 
$$a_k=q_k m_k^{\frac{n-\alpha}{n}}=w_k \hat{q}(\epsilon)m_k^{\frac{n-\alpha}{n}},$$
one obtains
$$
\frac{a_k}{m_k^{\frac{n-\alpha}{n}}}=w_k \hat{q}(\epsilon) \rightarrow 0
$$

$$
 |P_\alpha(\gamma_kE_{m_k}')-P_\alpha(\gamma_kE_{m_k})| \rightarrow 0
$$
as $k \rightarrow \infty$.
Hence
\begin{align}
\delta(\gamma_k E_{m_k}') & \le |\delta(\gamma_k E_{m_k}')-\delta(\gamma_k E_{m_k})|+\delta(\gamma_k E_{m_k}) \\
&= \frac{1}{P_\alpha(B_1)}|P_\alpha(\gamma_kE_{m_k}')-P_\alpha(\gamma_kE_{m_k})|+\delta(\gamma_k E_{m_k}) \label{wlkj2}\\
&\hskip .15in \rightarrow 0 
\end{align}
as $k \rightarrow \infty$.
Next, by the sharp stability of the nonlocal  isoperimetric inequality \cite{MR3322379} or a compactness argument, there exist $x_k, x_k' \in \mathbb{R}^n$ such that

\begin{equation} \label{wl5}
\frac{|(\gamma_kE_{m_k}+x_k) \Delta B_1|}{|\gamma_k E_{m_k}|} \rightarrow 0,
\end{equation}
\&
\begin{equation} \label{wlej}
\frac{|(\gamma_kE_{m_k}'+x_k') \Delta B_1|}{|\gamma_k E_{m_k}'|} \rightarrow 0
\end{equation}
as $k \rightarrow \infty$.\\

\noindent Therefore \eqref{wl5} and \eqref{wlej} yield $k \in \mathbb{N}$ such that
\begin{align*}
\frac{|(E_{m_k}'+\frac{(x_k'-x_k)}{\gamma_k}) \Delta E_{m_k}|}{|E_{m_k}|}&=\frac{|(\gamma_kE_{m_k}'+x_k')  \Delta (\gamma_kE_{m_k}+x) |}{|\gamma_k E_{m_k}|}\\
&\le \frac{|(\gamma_kE_{m_k}+x) \Delta B_1|}{|\gamma_k E_{m_k}|}+\frac{|B_1\Delta (\gamma_kE_{m_k}'+x_k')|}{|\gamma_k E_{m_k}|}\\
&<\epsilon,
\end{align*}
a contradiction to
\begin{align*}
\frac{\Big|\Big(E_{m_k}'+\frac{(x_k'-x_k)}{\gamma_k}\Big) \Delta E_{m_k}\Big|}{|E_{m_k}|}&=\frac{\Big|E_{m_k}' \Delta \Big(E_{m_k}-\frac{(x_k'-x_k)}{\gamma_k}\Big)\Big|}{|E_{m_k}|} \\
&\ge \inf_{A \in \mathcal{A}_{m_k}}  \frac{|E_{m_k}'\Delta AE_{m_k}|}{|E_{m_k}|}\\
& \ge \epsilon>0,
\end{align*}
thanks to  \eqref{wlje1}. To finish, assume $E_m$ is a minimizer, $m<m_0$, and define $\gamma_m=(\frac{|B_1|}{m})^{\frac{1}{n}}$,
$$|\gamma_m E_{m}|=|B_1|.$$  
Since
$$
P_\alpha(\frac{1}{\gamma_m}B_1) \le P_\alpha(E_m),
$$

\begin{align*}
P_\alpha(\frac{1}{\gamma_m}B_1)+\int_{E_m} g(x)dx &\le P_\alpha(E_m)+\int_{E_m} g(x)dx\\
&\le  P_\alpha(\frac{1}{\gamma_m}B_1)+\int_{\frac{1}{\gamma_m}B_1} g(x)dx,
\end{align*}
thus via subtracting
$$
P_\alpha(\frac{1}{\gamma_m}B_1)+\int_{E_m} g(x)dx,
$$
\begin{align} 
P_\alpha(E_m)-P_\alpha(\frac{1}{\gamma_m}B_1) &\le \int_{\frac{1}{\gamma_m}B_1} g(x)dx-\int_{E_m} g(x)dx \notag\\
&\le (\sup_{B_{R_m}} g) m, \label{rq4}
\end{align}
$$
\frac{1}{\gamma_m}B_1 \subset B_{R_m}.
$$
Next \eqref{rq4} implies
\begin{align}
P_\alpha(\gamma_mE_m)-P_\alpha(B_1)&=(\gamma_m)^{n-\alpha}(P_\alpha(E_m)-P_\alpha(\frac{1}{\gamma_m}B_1))\notag\\
& \le (\gamma_m)^{n-\alpha} (\sup_{B_{R_m}} g) m = |B_1|^{\frac{n-\alpha}{n}} (\sup_{B_{R_m}} g) m^{\frac{\alpha}{n}}.  \label{w1n}
\end{align}
Therefore, since 
$$
\delta(\gamma_m E_m) \ge c(n,\alpha)\Big(\frac{|\gamma_m E_m \Delta (a_m+B_1)|}{|B_1|} \Big)^2, 
$$
thanks to \eqref{w1n} and the sharp quantitative nonlocal isoperimetric inequality one obtains
\begin{align*}
c(n,\alpha)\Big(\frac{|\gamma_m E_m \Delta (a_m+B_1)|}{|B_1|} \Big)^2& \le \delta(\gamma_m E_m)\\
&=\frac{P_\alpha(\gamma_m E_m)-P_\alpha(B_1)}{P_\alpha(B_1)}\\
& \le \frac{|B_1|^{\frac{n-\alpha}{n}}}{P_\alpha(B_1)} (\sup_{B_{R_m}} g) m^{\frac{\alpha}{n}}.
\end{align*}
Now
\begin{equation} \label{vt9}
|\gamma_m E_m \Delta (a_{E_m}+B_1)| \le |B_1| \Big(\frac{1}{c(n,\alpha)} \frac{|B_1|^{\frac{n-\alpha}{n}}}{P_\alpha(B_1)} (\sup_{B_{R_m}} g)\Big)^{\frac{1}{2}} m^{\frac{\alpha}{2n}}.
\end{equation}

Observe this is valid for any minimizer $E_m$, therefore for any minimizer $A E_m$ with $A \in \mathcal{A}_m$: 
\begin{equation}  \label{vt99}
|\gamma_m AE_m \Delta (a_{AE_m}+B_1)| \le |B_1|\Big(\frac{1}{c(n, \alpha)} \frac{|B_1|^{\frac{n-\alpha}{n}}}{P_\alpha(B_1)} (\sup_{B_{R_m}} g)\Big)^{\frac{1}{2}} m^{\frac{\alpha}{2n}}.
\end{equation}
Therefore \eqref{vt9}, \eqref{vt99} yield
\begin{align*}
& |\gamma_m E_m \Delta \Big(\gamma_m AE_m +(a_{E_m}-a_{AE_m})\Big)| \\
&\le |\gamma_m E_m \Delta (a_{E_m}+B_1)|+|(a_{E_m}+B_1)\Delta \Big(\gamma_m AE_m +(a_{E_m}-a_{AE_m})\Big)|\\
&=|\gamma_m E_m \Delta (a_{E_m}+B_1)|+|(a_{AE_m}+B_1)\Delta \Big(\gamma_m AE_m \Big)|\\ 
&\le 2 |B_1|\Big(\frac{1}{c(n,\alpha)} \frac{|B_1|^{\frac{n-\alpha}{n}}}{P_\alpha(B_1)} (\sup_{B_{R_m}} g)\Big)^{\frac{1}{2}} m^{\frac{\alpha}{2n}}.
\end{align*}
Hence
\begin{equation} \label{ra1k}
\frac{|B_1|}{m} |E_m \Delta \Big(AE_m +\frac{(a_{E_m}-a_{AE_m})}{\gamma_m}\Big)| \le 2 |B_1|\Big(\frac{1}{c(n,\alpha)} \frac{|B_1|^{\frac{n-\alpha}{n}}}{P_\alpha(B_1)} (\sup_{B_{R_m}} g)\Big)^{\frac{1}{2}} m^{\frac{\alpha}{2n}}.
\end{equation}
Set
$$
\alpha_m:=\frac{a_{AE_m}-a_{E_m}}{\gamma_m},
$$
thus via \eqref{ra1k}, 
\begin{align} \label{appr}
|AE_m\Delta \Big(E_m + \alpha_m \Big)| &\le 2 \Big(\frac{1}{c(n,\alpha)} \frac{|B_1|^{\frac{n-\alpha}{n}}}{P_\alpha(B_1)} (\sup_{B_{R_m}} g)\Big)^{\frac{1}{2}} m^{1+\frac{\alpha}{2n}} \notag\\
&=2 \Big(\frac{1}{c(n,\alpha)}  \frac{|B_1|^{\frac{n-\alpha}{n}}}{P_\alpha(B_1)} (\sup_{B_{R_m}} g)\Big)^{\frac{1}{2}} m^{1+\frac{\alpha}{2n}}.
\end{align}

\end{proof}

\begin{rem}
The theorem may also be extended to $g \in L_{loc}^1(\{g<\infty\})$ with some assumptions via Lebesgue's differentiation theorem.
\end{rem}

\bibliographystyle{amsalpha}
\bibliography{References}

\newcommand{\etalchar}[1]{$^{#1}$}
\providecommand{\bysame}{\leavevmode\hbox to3em{\hrulefill}\thinspace}
\providecommand{\MR}{\relax\ifhmode\unskip\space\fi MR }
\providecommand{\MRhref}[2]{%
  \href{http://www.ams.org/mathscinet-getitem?mr=#1}{#2}
}
\providecommand{\href}[2]{#2}
\begin{thebibliography}{FFM{\etalchar{+}}15}

\bibitem[BNO23]{MR4674821}
Konstantinos Bessas, Matteo Novaga, and Fumihiko Onoue, \emph{On the shape of
  small liquid drops minimizing nonlocal energies}, ESAIM Control Optim. Calc.
  Var. \textbf{29} (2023), Paper No. 86, 26. \MR{4674821}

\bibitem[Bre91]{pFZ4}
Y.~Brenier, \emph{Polar factorization and monotone rearrangement of
  vector-valued functions}, Commun. Pure Appl. Math. \textbf{44 (4)} (1991),
  375--417.

\bibitem[CN17]{MR3640534}
Annalisa Cesaroni and Matteo Novaga, \emph{Volume constrained minimizers of the
  fractional perimeter with a potential energy}, Discrete Contin. Dyn. Syst.
  Ser. S \textbf{10} (2017), no.~4, 715--727. \MR{3640534}

\bibitem[CR24]{qkvk}
O.~Chodosh and I.~Ruohoniemi, \emph{On minimizers in the liquid drop model},
  arXiv:2401.04822v1 (2024).

\bibitem[CRS10]{MR2675483}
L.~Caffarelli, J.-M. Roquejoffre, and O.~Savin, \emph{Nonlocal minimal
  surfaces}, Comm. Pure Appl. Math. \textbf{63} (2010), no.~9, 1111--1144.
  \MR{2675483}

\bibitem[Cur85]{Crist}
P.~Curie, \emph{Sur la formation des cristaux et sur les constantes capillaires
  de leurs different faces}, Bulletin de la Societe Francaise de Mineralogie et
  de Cristallographie \textbf{8} (1885), 145--150.

\bibitem[DPV22]{MR4380032}
Serena Dipierro, Giorgio Poggesi, and Enrico Valdinoci, \emph{Radial symmetry
  of solutions to anisotropic and weighted diffusion equations with
  discontinuous nonlinearities}, Calc. Var. Partial Differential Equations
  \textbf{61} (2022), no.~2, Paper No. 72, 31. \MR{4380032}

\bibitem[FFM{\etalchar{+}}15]{MR3322379}
A.~Figalli, N.~Fusco, F.~Maggi, V.~Millot, and M.~Morini, \emph{Isoperimetry
  and stability properties of balls with respect to nonlocal energies}, Comm.
  Math. Phys. \textbf{336} (2015), no.~1, 441--507. \MR{3322379}

\bibitem[FI13]{MR3023863}
Alessio Figalli and Emanuel Indrei, \emph{A sharp stability result for the
  relative isoperimetric inequality inside convex cones}, J. Geom. Anal.
  \textbf{23} (2013), no.~2, 938--969. \MR{3023863}

\bibitem[Fig13]{zbMATH06728315}
Alessio Figalli, \emph{Stability in geometric and functional inequalities},
  European Congress of Mathematics. Proceedings of the 6th ECM congress,
  Krak\'ow, Poland, July 2--7 July, 2012, Z{\"u}rich: European Mathematical
  Society (EMS), 2013, pp.~585--599 (English).

\bibitem[FLS08]{MR2425175}
Rupert~L. Frank, Elliott~H. Lieb, and Robert Seiringer,
  \emph{Hardy-{L}ieb-{T}hirring inequalities for fractional {S}chr\"odinger
  operators}, J. Amer. Math. Soc. \textbf{21} (2008), no.~4, 925--950.
  \MR{2425175}

\bibitem[FM91]{MR1130601}
Irene Fonseca and Stefan M\"{u}ller, \emph{A uniqueness proof for the {W}ulff
  theorem}, Proc. Roy. Soc. Edinburgh Sect. A \textbf{119} (1991), no.~1-2,
  125--136. \MR{1130601}

\bibitem[FM11]{MR2807136}
A.~Figalli and F.~Maggi, \emph{On the shape of liquid drops and crystals in the
  small mass regime}, Arch. Ration. Mech. Anal. \textbf{201} (2011), no.~1,
  143--207. \MR{2807136}

\bibitem[FMP08]{MR2456887}
N.~Fusco, F.~Maggi, and A.~Pratelli, \emph{The sharp quantitative isoperimetric
  inequality}, Ann. of Math. (2) \textbf{168} (2008), no.~3, 941--980.
  \MR{2456887}

\bibitem[FMP10]{MR2672283}
A.~Figalli, F.~Maggi, and A.~Pratelli, \emph{A mass transportation approach to
  quantitative isoperimetric inequalities}, Invent. Math. \textbf{182} (2010),
  no.~1, 167--211. \MR{2672283}

\bibitem[FN21]{MR4314139}
Rupert~L. Frank and Phan~Th\`anh Nam, \emph{Existence and nonexistence in the
  liquid drop model}, Calc. Var. Partial Differential Equations \textbf{60}
  (2021), no.~6, Paper No. 223, 12. \MR{4314139}

\bibitem[Fon91]{MR1116536}
Irene Fonseca, \emph{The {W}ulff theorem revisited}, Proc. Roy. Soc. London
  Ser. A \textbf{432} (1991), no.~1884, 125--145. \MR{1116536}

\bibitem[FZ22]{pFZ}
Alessio Figalli and Yi~Ru-Ya Zhang, \emph{Strong stability for the wulff
  inequality with a crystalline norm}, Communications on Pure and Applied
  Mathematics \textbf{75} (2022), no.~2, 422--446.

\bibitem[Gam30]{row52}
G.~Gamow, \emph{Mass defect curve and nuclear constitution}, Proc. R. Soc.
  Lond. Ser. A \textbf{126} (1930), 632--644.

\bibitem[Gib78]{G}
J.W. Gibbs, \emph{On the equilibrium of heterogeneous substances}, Collected
  Works \textbf{1} (1878).

\bibitem[Hal92]{zbMATH00055774}
R.~R Hall, \emph{A quantitative isoperimetric inequality in
  {{\(n\)}}-dimensional space}, J. Reine Angew. Math. \textbf{428} (1992),
  161--175 (English).

\bibitem[IK23]{qk}
E.~Indrei and A.~Karakhanyan, \emph{Minimizing the free energy},
  arXiv:2304.01866 (2023).

\bibitem[IK24]{qkv}
\bysame, \emph{On the three-dimensional shape of a crystal}, arXiv:2406.00241
  (2024).

\bibitem[Ind21]{I20}
Emanuel Indrei, \emph{A weighted relative isoperimetric inequality in convex
  cones}, Methods Appl. Anal. 27 \textbf{28} (2021), no.~1, 001--014.

\bibitem[Ind23]{qkvp}
E~Indrei, \emph{Small mass uniqueness in the anisotropic liquid drop model and
  the critical mass problem}, arXiv:2311.16573 (2023).

\bibitem[Ind24]{MR4730410}
Emanuel Indrei, \emph{On the equilibrium shape of a crystal}, Calc. Var.
  Partial Differential Equations \textbf{63} (2024), no.~4, Paper No. 97, 33.
  \MR{4730410}

\bibitem[Pog24]{MR1897393g}
Giorgio Poggesi, \emph{Soap bubbles and convex cones: optimal quantitative
  rigidity}, Trans. Amer. Math. Soc. (2024).

\bibitem[TA87]{MR872883}
Jean~E. Taylor and F.~J. Almgren, Jr., \emph{Optimal crystal shapes},
  Variational methods for free surface interfaces ({M}enlo {P}ark, {C}alif.,
  1985), Springer, New York, 1987, pp.~1--11. \MR{872883}

\bibitem[Tay78]{MR493671}
Jean~E. Taylor, \emph{Crystalline variational problems}, Bull. Amer. Math. Soc.
  \textbf{84} (1978), no.~4, 568--588. \MR{493671}

\end{thebibliography}
\end{document}